\documentclass{amsart}
\usepackage{enumerate, amssymb, amsmath, rotating, comment}
\usepackage{color}

\def\dist{\operatorname{dist}}
\def\er{\mathbb R}
\def\N{\mathbb N}
\def\C{\mathbb C}

\def\eps{\varepsilon}

\def\hd{\hat{\operatorname{d}}}
\def\wk{\operatorname{wk}}
\def\wck{\operatorname{wck}}
\def\ca{\operatorname{ca}}
\def\uc{\operatorname{uc}}

\def\sa{_{\operatorname{sa}}}
\def\q#1{(#1$_q$)}
\def\qo#1{(#1$_q$)$_\omega$}
\def\qos#1{(#1$_q$)$_\omega^*$}

\def\clust{\operatorname{clust}}

\def\sgn{\operatorname{sgn}}

\def\dws{{\delta_{w^*}}}

\def\trn#1{{\left\vert\kern-0.3ex\left\vert\kern-0.3ex\left\vert\hskip 0,1 mm #1\hskip 0,1 mm\right\vert\kern-0.3ex\right\vert\kern-0.3ex\right\vert}}

\newtheorem{theorem}{Theorem}[section]

\newtheorem*{lemma*}{Lemma}
\newtheorem{corollary}[theorem]{Corollary}

\theoremstyle{definition}
\newtheorem*{remark}{Remark}
\newtheorem*{definition}{Definition}
\newtheorem*{claim}{Claim}

\begin{document}

\title{C*-algebras have a quantitative version of Pe\l czy\' nski's property (V)}
\author{Hana Kruli\v sov\'a}

\address{Department of Mathematical Analysis \\
Faculty of Mathematics and Physics\\ Charles University in Prague\\
Sokolovsk\'{a} 83, 186 \ 75\\Praha 8, Czech Republic}

\email{krulisova@karlin.mff.cuni.cz}


\thanks{The research was supported by the Grant No. 142213/B-MAT/MFF of the Grant Agency of the Charles University in Prague and by the Research grant GA\v{C}R P201/12/0290.}

\begin{abstract}
A Banach space $X$ has Pe\l czy\' nski's property (V) if for every Banach space $Y$ every unconditionally converging operator $T\colon X\to Y$ is weakly compact. H. Pfitzner proved that $C^*$-algebras have Pe\l czy\' nski's property (V). In the preprint \cite{qpel} the author explores possible quantifications of the property (V) and shows that $C(K)$ spaces for a compact Hausdorff space $K$ enjoy a quantitative version of the property (V). In this paper we generalize this result by quantifying Pfitzner's theorem. Moreover, we prove that in dual Banach spaces a quantitative version of the property (V) implies a quantitative version of the Grothendieck property.
\end{abstract}

\maketitle

\section{Introduction}
In 1994, H. Pfitzner proved that $C^*$-algebras have Pe\l czy\' nski's property (V) (see \cite{pfitzner}). The aim of this paper is to prove a quantitative version of Pfitzner's result. In this way we continue the study of quantitative versions of Pe\l czy\' nski's property (V) presented in the preprint \cite{qpel}.

Section \ref{preliminaries} summarizes all essential definitions and basic facts contained mostly in the preprint \cite{qpel}. In Section \ref{vylepseny-behrends} we slightly improve Behrends's quantitative version of Rosenthal's $\ell^1$--theorem \cite[Section 3]{behrends}, which we use to prove the main theorem in Section \ref{main}. Section \ref{groth} is devoted to the relationship of quantitative versions of Pe\l czy\' nski's property (V) and the Grothendieck property in dual Banach spaces.

\section{Preliminaries}
\label{preliminaries}
We follow the notation of \cite{qpel} with one exception. Because we deal also with $C^*$\mbox{-}algebras, we write $X'$ (instead of $X^*$) for a dual to a Banach space $X$, since the $^*$ in $C^*$\mbox{-}algebras is already reserved for the involution. All Banach spaces are considered either real or complex, unless stated otherwise. The closed unit ball of a Banach space $X$ is denoted by $B_X$.

\subsection{Pe\l czy\' nski's property (V) and its quantification}
Let us recall some essential definitions and facts (explained in more detail in \cite{qpel} with many comments). A series $\sum_{n=1}^\infty x_n$ in a Banach space $X$ is said to be
\begin{itemize}
\item \emph{unconditionally convergent}\/ if the series $\sum_{n=1}^\infty {t_n}x_n$ converges whenever $(t_n)$ is a bounded sequence of scalars,
\item \begin{sloppypar}\emph{weakly unconditionally Cauchy} (\emph{wuC}) if for all $x'\in X'$ the series $\sum_{n=1}^\infty |x'(x_n)|$ converges.\end{sloppypar}
\end{itemize}
A bounded linear operator $T\colon X\to Y$ between Banach spaces $X$ and $Y$ is called \emph{unconditionally converging} if $\sum_{n} Tx_n$ is an unconditionally convergent series in $Y$ whenever $\sum_{n} x_n$ is a weakly unconditionally Cauchy series in $X$. It is not difficult to show that $T$ is unconditionally converging if and only if for every series $\sum_{n} x_n$ in $X$ with
$$\sup_{x'\in B_{X'}} \sum_{n=1}^\infty |x'(x_n)| < \infty$$
the series $\sum_{n} Tx_n$ converges.
We say that a Banach space $X$ has \emph{Pe\l czy\' nski's property (V)} if for every Banach space $Y$ every unconditionally converging operator $T\colon X\to Y$ is weakly compact.

To quantify the property (V) means to replace the implication
\begin{equation}
\label{impl}
T \text{ is unconditionally converging } \Rightarrow\ T \text{ is weakly compact}
\end{equation}
by an inequality
$$\begin{aligned} &\text{measure of weak non-compactness of } T \\
&\hskip 3 cm \leq C\cdot\text{measure of } T \text{ not being unconditionally converging},
\end{aligned}$$
where $C$ is some positive constant depending only on $X$, and the two measures are positive numbers for each operator $T$ and are equal to zero if and only if $T$ is weakly compact or unconditionally converging, respectively. This inequality is a~strengthening of the original implication \eqref{impl}.

For this purpose we use the following quantities. For a bounded sequence $(x_n)$ in a Banach space $X$ we define
$$\ca\big((x_n)\big)=\inf_{n\in\N} \sup \{\|x_k-x_l\|:\ k,l\in\N,\ k,l\geq n\}.$$
It is a measure of non-Cauchyness of a sequence $(x_n)$, hence in Banach spaces it measures non-convergence.
Let $T\colon X \to Y$ be a bounded linear operator between Banach spaces $X$ and $Y$. We set
$$\uc(T) = \sup\left\{\ca\left(\bigg(\sum_{i=1}^n T x_i\bigg)_n\right):\ (x_n)\subset X,\ \sup_{x'\in B_{X'}}\sum_{n=1}^\infty |x'(x_n)|\leq 1\right\}.$$
Then $\uc(T)$ measures how far is the operator $T$ from being unconditionally converging.

Let $A$ be a bounded subset of a Banach space $X$. The de Blasi measure of weak non-compactness of the set $A$ is defined by
$$\omega(A)=\inf\{\hd(A,K): \emptyset\not=K\subset X \text{ is weakly compact}\},$$
where
$$\hd(A,K) = \sup\{\dist(a,K):\ a\in A\}.$$
De Blasi has proved that $\omega(A)=0$ if and only if $A$ is relatively weakly compact (see \cite{deblasi}).
Other quantities which measure relative weak non-compactness are for example
$$
\begin{aligned}
\gamma(A) &= \sup\{|\lim_n\lim_m x'_m(x_n) - \lim_m\lim_n x'_m(x_n) |:\ (x_n) \text{ is a sequence in } A,\\
&\hskip 35 mm (x'_m) \text{ is a sequence in } B_{X'}, \text{ and the limits exist}\}
\end{aligned}
$$
or
$$\wck_X(A) = \sup\{\dist(\clust_{(X'',w^*)}(x_n),X):\ (x_n) \text{ is a sequence in } A\},$$
where $\clust_{(X'',w^*)}(x_n)$ stands for the set of all $w^*$-cluster points of the sequence $(x_n)$ in $X''$.
The quantities $\gamma(A)$ and $\wck_X(A)$ are equivalent by \cite[Theorem 2.3]{ac} in the following sense:
\begin{equation}
\label{equiv}
\wck_X(A) \leq \gamma(A) \leq 2 \wck_X(A).
\end{equation}
However, the quantity $\omega(A)$ is not equivalent to the other two (see \cite[Corollary 3.4]{ac}). We have only
\begin{equation}
\wck_X(A) \leq \omega(A)
\end{equation}
by \cite[Theorem 2.3]{ac}.

For measuring weak non-compactness of a bounded linear operator $T\colon X \to Y$ between Banach spaces $X$ and $Y$ we use the quantities $\omega(T(B_X))$, $\gamma(T(B_X))$, and $\wck_Y(T(B_X))$, which we denote simply by $\omega(T)$, $\gamma(T)$, and $\wk_Y(T)$.

We say that a Banach space $X$ has a quantitative version of Pe\l czy\' nski's property (V) -- we denote it by \q{V} -- if there is a~constant $C>0$ such that for every Banach space $Y$ and every operator $T\colon X\to Y$
\begin{equation}
\label{quantpel}
\gamma(T)\leq C\cdot \uc(T).
\end{equation}
If it is possible to replace $\gamma(T)$ in \eqref{quantpel} with $\omega(T)$, we say that $X$ has the property \qo{V}. If $\gamma(T)$ in \eqref{quantpel} is replaced by $\omega(T')$, where $T'\colon Y'\to X'$ denotes the dual operator to $T$, we say that $X$ has the property \qos{V}.

\begin{sloppypar}
In \cite[Proposition 3.2]{qpel} it is proved that a Banach space $X$ has the property \q{V} if and only if there exists a constant $C>0$ such that for each bounded subset $K$ of the dual space $X'$
$$\gamma(K) \leq C\cdot \eta(K),$$
where $$\eta(K)=\sup\Big\{\limsup_n \sup_{x'\in K} |x'(x_n)|:\,(x_n)\subset X,\,\sup_{x'\in B_{X'}}\sum_{n=1}^\infty |x'(x_n)|\leq 1\Big\}.$$
Using the above-described characterization we will prove in Section \ref{main} that \mbox{$C^*$-algebras} have the property \q{V}.
\end{sloppypar}

Note that the quantity $\eta$ is translation-invariant, that is,
\begin{equation}\label{eta-transl}
\eta(K) = \eta(K+z'),\quad K\subset X',\ z'\in X'.
\end{equation}
This follows from the fact that $(x_n)$ weakly null whenever $\sum x_n$ is a wuC series in $X$.

\subsection{Measures of weak and weak$^*$ non-Cauchyness of sequences in Banach spaces}
In sections \ref{main} and \ref{groth} we will use the following standard quantities, analogous to the quantity $\ca$, which measure how far is a sequence in a (dual) Banach space from being weakly (weak$^*$) Cauchy.

Let $X$ be a Banach space and let $(x_n)$ be a bounded sequence $X$. We set
$$\delta(x_n) = \sup_{x'\in B_{X'}} \lim_{n\to\infty} \sup_{k,l\geq n} |x'(x_k') - x'(x_l')|.$$
This quantity is a measure of weak non-Cauchyness of the sequence $(x_n)$.
Furthermore, let us set
$$\tilde{\delta}(x_n) = \inf\left\{\delta(x_{n_k}):\,(x_{n_k}) \text{ is a subsequence of } (x_n)\right\}.$$
It measures how close can subsequences of $(x_n)$ be to be weakly Cauchy.

If $(x_n')$ is a bounded sequence in $X'$, we set
$$\dws(x_n') = \sup_{x\in B_{X}} \lim_{n\to\infty} \sup_{k,l\geq n} |x_k'(x) - x_l'(x)|.$$
The last quantity is a measure of weak$^*$ non-Cauchyness of the sequence $(x_n')$. The quantity $\delta(x_n)$ equals $0$ if and only if the sequence $(x_n)$ is weakly Cauchy. Analogously, $\dws(x_n')=0$ if and only if $(x_n')$ is weak$^*$ Cauchy. If $\tilde{\delta}(x_n)=0$, it it not clear whether $(x_n)$ admits a weakly Cauchy subsequence.

\subsection{Selfadjoint elements and selfadjoint functionals}

Let $A$ be a $C^*$-algebra. Let us denote by $A\sa$ the selfadjoint elements of $A$, that is $A\sa = \{a\in A:\, a=a^*\}$. Then $A\sa$ is a real Banach space and $A=A\sa + i A\sa$. If $f$ is a bounded linear functional on $A$, $f^*$ is the functional defined by $f^*(x) = \overline{f(x^*)}$, $x\in A$. Let $(A')\sa$ denote the set $\{f\in A':\, f=f^*\}$ of selfadjoint functionals on $A$. Then $(A')\sa$ is a real Banach space, and is isometrically isomorphic to $(A\sa)'$. We will write $A'\sa$ for both these spaces. Every functional $x'\in A'$ can be decomposed as $x'=f+ig$ where $f,g\in A'\sa$. It suffices to set $f=(x'+(x')^*)/2$, $g=(x'-(x')^*)/(2i)$.

\section{A quantitative version of Rosenthal's $\ell^1$--theorem}
\label{vylepseny-behrends}
For proving the main result we need the quantitative version of Rosenthal's $\ell^1$--theorem proved by E. Behrends in \cite[Section 3]{behrends}. In this section we revise his theorem, because it turns out that one of the estimates there can be easily improved. We will then use this improved version.

Let us remind Behrends's definition \cite[3.1]{behrends}.

\begin{definition}
Let $(x_n)$ be a bounded sequence in a Banach space $X$ and $\eps>0$. We say that $(x_n)$ admits $\eps$--$\ell^1$--blocks if for every infinite $M\subset \N$ there are scalars $a_1,\dots,a_r$ with $\sum|a_r|=1$ and $i_1,\dots,i_r$ in $M$ such that $\left\|\sum a_\rho x_{i_\rho}\right\| \leq \eps$.
\end{definition}

The revised version of the quantitative Rosenthal's $\ell^1$--theorem for complex Banach spaces is the following.

\begin{theorem}
\label{rosenthal-complex}
Let $X$ be a complex Banach space $X$ and $\eps>0$. Let $(x_n)$ be a~sequence in $X$ which admits $\eps$--$\ell^1$--blocks. Then there is a subsequence $(x_{n_k})$ of $(x_n)$ such that for every $x'\in X'$ with $\|x'\|=1$ the diameter of the set of cluster points of the sequence $(x'(x_{n_k}))_k$ is at most $\pi\eps$.
\end{theorem}

\begin{remark}
In the original Behrends' theorem \cite[Theorem 3.3]{behrends} there is a larger constant $8/\sqrt 2$ in place of $\pi$. A similar result with the better constant $\pi$ has been obtained (in a different way) by I.~Gasparis \cite{gasparis}.
\end{remark}

\begin{proof}[Sketch of the proof of Theorem \ref{rosenthal-complex}]
The proof is essentially the same as the original one. Suppose that the conclusion were not true. We can find $\delta>0$ such that the number
$$\sup_{x'\in S_{X'}} \left\{\text{diameter of the set of accumulation points of } (x'(x_{n_k}))_k \right\}$$
is greater than $\pi\eps + \delta$ for any subsequence $(x_{n_k})$ of $(x_n)$. Fix $\tau\in(0,1)$ such that $(2+\sup_n \|x_n\|)\tau < \frac{\delta}{\pi}$.

Similarly to the one in the proof of \cite[Theorem 3.3 (or 3.2)]{behrends} we can prove the following lemma.
\begin{lemma*}
The sequence $(x_n)$ admits a subsequence (without loss of generality still denoted by $(x_n)$) which satisfies the following conditions:
\begin{enumerate}[(i)]
\item Whenever $C$ and $D$ are disjoint finite subsets of\/ $\N$, there are $z_0, w_0 \in \C$ with $|w_0| \geq \pi\eps + \delta$ and $x'\in X'$ with $\|x'\|=1$ such that $|x'(x_n) - z_0| \leq \tau$ for $n\in C$ and $|x'(x_n) - (z_0+w_0)| \leq \tau$ for $n\in D$.
\item There are $i_1<\dots<i_r$ in\/ $\N$ and $a_1,\dots,a_r\in\C$ which satisfy
$$\sum_{\rho=1}^r |a_\rho|=1,\ \left|\sum_{\rho=1}^r a_\rho \right|\leq \tau,\ \left\|\sum_{\rho=1}^r a_\rho x_{i_\rho}\right\|\leq \eps. $$
\end{enumerate}
\end{lemma*}

Finally, the time has come for the modification. By \cite[Lemma 6.3]{rudin} we find $D \subset \{1,\dots,r\}$ such that
$$\left|\sum_{\rho\in D} a_\rho\right| \geq \frac{1}{\pi}\sum_{\rho=1}^r |a_\rho| = \frac{1}{\pi}.$$
Set $C=\{1,\dots,r\} \setminus D$. For these sets $C$ and $D$ we find $z_0$, $w_0$, and $x'$ from (i) of the lemma. It follows that
$$\begin{aligned}
\eps &\geq \left\|\sum_{\rho=1}^r a_\rho x_{i_\rho} \right\| \geq \left|\sum_{\rho=1}^r a_\rho x'(x_{i_\rho}) \right|
= \left|\sum_{\rho\in C} a_\rho x'(x_{i_\rho}) + \sum_{\rho\in D} a_\rho x'(x_{i_\rho}) \right| \\
&\geq \left|\sum_{\rho\in C} a_\rho z_0 + \sum_{\rho\in D} a_\rho (z_0+w_0) \right| - \tau\sum_{\rho=1}^r |a_\rho|
= \left|\sum_{\rho\in D} a_\rho w_0 + \sum_{\rho=1}^r a_\rho z_0 \right| - \tau \\
&\geq |w_0|\left|\sum_{\rho\in D} a_\rho\right| - |z_0|\left| \sum_{\rho=1}^r a_\rho \right| - \tau
\geq \frac{|w_0|}{\pi} - |z_0|\tau - \tau
\geq \frac{\pi\eps + \delta}{\pi} - (1+|z_0|)\tau \\
&= \eps + \frac{\delta}{\pi} - (1+|z_0|)\tau \geq \eps + \frac{\delta}{\pi} - (2+\sup_n \|x_n\|)\tau > \eps,
\end{aligned}$$
which is a contradiction. 
\end{proof}

\section{Main theorem}
\label{main}
This section is devoted to our main result -- a quantitative version of Pfitzner's theorem (Theorem \ref{qpf} below). We also prove a ``real version'' of this theorem (Theorem \ref{sa}).

\begin{theorem}
\label{qpf}
Let $A$ be a $C^*$-algebra. Then for every bounded $K\subset A'$
\begin{equation}
\label{qp}
\wck_{A'}(K) \leq \pi \cdot \eta(K).
\end{equation}
Therefore $A$ has the property \q{V}.
\end{theorem}
\begin{proof}
The quantities $\gamma(K)$ and $\wck_{A'}(K)$ are equivalent by \cite[Theorem 2.3]{ac}, more specifically, the inequality \eqref{qp} implies $\gamma(K) \leq 2\pi \cdot \eta(K)$. If this holds for each bounded $K\subset A'$, Proposition \cite[3.2]{qpel} mentioned also in Section \ref{preliminaries} gives that $A$ has the property \q{V}. Let us show the inequality \eqref{qp}.

Let $K\subset A'$ be bounded. The case $\wck_{A'}(K) = 0$ is trivial. Suppose that $\wck_{A'}(K) > 0$ and fix an arbitrary
$\lambda\in(0,\wck_{A'}(K))$. By the definition of the quantity $\wck_{A'}(K)$ we find a sequence $(x_n')$ in $K$ such that
$$\dist\left(\clust_{(A''',w^*)}(x_n'),A'\right)>\lambda.$$
Since every dual of a $C^*$-algebra is a predual of a von Neumann algebra, we deduce from \cite[Theorem III.2.14]{takesaki} (see also \cite[Example IV.1.1(b)]{mideals}) that $A'$ is L-embedded -- it means that $A'$ is complemented in $A'''$ by a projection $P$ satisfying
$$\|x'''\| = \|Px'''\| + \|x'''-Px'''\|, \quad x'''\in A'''.$$
Consequently, from \cite[Theorem 1]{kps} we have
$$\begin{aligned}
\tilde{\delta}(x_n') &= \inf\{\delta(x_{n_k}'):\,(x_{n_k}') \text{ is a subsequence of } (x_k')\} \\
&\geq 2\dist\left(\clust_{(A''',w^*)}(x_n'),A'\right) > 2\lambda.
\end{aligned}$$

Fix an arbitrary $\eps>0$. We now prove the following claim.
\begin{claim}
There is a sequence of self-adjoint elements $(x_k)$ in $B_A$ satisfying $x_i x_j = 0$, $i,j\in\N$, $i\not=j$, and a subsequence $(x_{n_k}')$ of the sequence $(x_n')$ such that
$$\left|x_{n_k}'(x_k)\right| > (1-\eps)^2 \frac{\lambda}{\pi}, \quad k\in\N.$$
\end{claim}
\begin{proof}\renewcommand{\qedsymbol}{}
Each $x_n'$ is canonically decomposed in the following way: $x_n'=f_n+ i g_n$, where $f_n,g_n\in A'$ are selfadjoint functionals. It suffices to find $(x_k)$ and $(x_{n_k}')$ such that
$$\left|f_{n_k}(x_k)\right| > (1-\eps)^2 \frac{\lambda}{\pi} \qquad \text{or} \qquad \left|g_{n_k}(x_k)\right| > (1-\eps)^2 \frac{\lambda}{\pi}.$$
Indeed, since selfadjoint functionals attain real values on selfadjoint elements of $A$, we have
$$\begin{aligned}
\left|x_{n_k}'(x_k)\right| &= \left|f_{n_k}(x_k) + i g_{n_k}(x_k)\right|
\geq
\begin{cases}
\big|\operatorname{Re}(f_{n_k}(x_k) + i g_{n_k}(x_k))\big| = |f_{n_k}(x_k)|\\
\big|\operatorname{Im}(f_{n_k}(x_k) + i g_{n_k}(x_k))\big| = |g_{n_k}(x_k)|
\end{cases}.
\end{aligned}$$

We begin by proving that there is a strictly increasing sequence of indices $(n_k)$ such that $\tilde{\delta}(f_{n_k})>\lambda$ or $\tilde{\delta}(g_{n_k})>\lambda$. If $\tilde{\delta}(f_n)>\lambda$, the proof is over, so suppose that $\tilde{\delta}(f_n)\leq\lambda$. Let us find $\tau>0$ satisfying $\tilde{\delta}(x_n')>2\lambda+2\tau$. By the definition of $\tilde{\delta}(f_n)$ there is a subsequence $(f_{n_k})$ of the sequence $(f_n)$ with $\delta(f_{n_k}) < \lambda + \tau$. We claim that the corresponding subsequence $(g_{n_k})$ of $(g_n)$ satisfies $\tilde{\delta}(g_{n_k}) > \lambda$. To obtain a contradiction, suppose that $\tilde{\delta}(g_{n_k}) \leq \lambda$. Using the definition of $\tilde{\delta}(g_{n_k})$ we find a~strictly increasing sequence of indices $(k_l)$ such that $\delta(g_{n_{k_l}}) < \lambda + \tau$. Then
$$\begin{aligned}
\delta(x_{n_{k_l}}') &= \delta(f_{n_{k_l}} + i g_{n_{k_l}}) \\
&= \sup_{x''\in B_{A''}} \lim_{l\to\infty} \sup_{p,q\geq l} \big|x''(f_{n_{k_p}} + i g_{n_{k_p}}) - x''(f_{n_{k_q}} + i g_{n_{k_q}})\big| \\
&\leq \sup_{x''\in B_{A''}} \lim_{l\to\infty} \sup_{p,q\geq l} \left(\big|x''(f_{n_{k_p}}) - x''(f_{n_{k_q}})\big| + \big|x''(g_{n_{k_p}}) - x''(g_{n_{k_q}})\big|\right) \\
&\leq \sup_{x''\in B_{A''}} \lim_{l\to\infty} \sup_{p,q\geq l} \big|x''(f_{n_{k_p}}) - x''(f_{n_{k_q}})\big| \\
&\qquad+ \sup_{x''\in B_{A''}} \lim_{l\to\infty} \sup_{p,q\geq l} \big|x''(g_{n_{k_p}}) - x''(g_{n_{k_q}})\big| \\
&= \delta(f_{n_{k_l}}) + \delta(g_{n_{k_l}}) < \lambda + \tau + \lambda + \tau = 2\lambda + 2\tau,
\end{aligned}$$
which contradicts the fact that $\tilde{\delta}(x_n') > 2\lambda + 2\tau$.

Without loss of generality we may assume that we have found a subsequence $(f_{n_k})$ of the sequence $(f_n)$ with $\tilde{\delta}(f_{n_k})>\lambda$ and such that $(f_{n_k}) = (f_n)$. By passing to a~further subsequence we can also ensure that
$$\frac{\inf_{n\in\N} \|f_n\|}{\sup_{n\in\N} \|f_n\|} > 1-\eps.$$
Indeed, the sequence $(f_n)$ is bounded, hence we can find its subsequence $(f_{n_k})$ such that the $\lim_{k\to\infty} \|f_{n_k}\|$ exists. This limit is nonzero, because otherwise we would have $\tilde{\delta}(f_n)=0$. We thus obtain the desired subsequence by omitting finitely many members of $(f_{n_k})$.

The inequality $\tilde{\delta}(f_n)>\lambda$ says that for every subsequence $(f_{n_k})$ of $(f_n)$ there is some $x''\in A''$ with $\|x''\|=1$ such that the diameter of the set of accumulation points of the sequence $(x''(f_{n_k}))_k$ is greater than $\lambda$. By Theorem \ref{rosenthal-complex} the sequence $(f_n)$ does not admit $\frac{\lambda}{\pi}$--$\ell^1$--blocks, i.e. there is an infinite $M\subset\N$ such that whenever $a_1,\dots,a_r\in\C$ satisfy $\sum_{i=1}^r |a_i| = 1$, and $n_1<\dots<n_r$ are indices in $M$, we have $\big\|\sum_{i=1}^r a_i f_{n_i} \big\| > \frac{\lambda}{\pi}$. Hence there is a subsequence $(f_{n_k})$ of $(f_n)$ such that for each nonzero $(\alpha_k)\in\ell^1$ and $N\in\N$ large enough
$$\left\|\sum_{k=1}^N \frac{\alpha_k}{\sum_{k=1}^N |\alpha_k|} f_{n_k}\right\| > \frac{\lambda}{\pi}.$$
By letting $N\to\infty$ we obtain
$$\frac{\lambda}{\pi} \sum_{k=1}^\infty |\alpha_k| \leq \left\|\sum_{k=1}^\infty \alpha_k f_{n_k}\right\|.$$
Therefore we have for each $(\alpha_k)\in\ell^1$
$$\frac{\lambda}{\pi \sup_{k\in\N} \|f_{n_k}\|} \sum_{k=1}^\infty |\alpha_k|
\leq \frac{\lambda}{\pi} \sum_{k=1}^\infty \frac{|\alpha_k|}{\|f_{n_k}\|}
\leq \left\|\sum_{k=1}^\infty \alpha_k \frac{f_{n_k}}{\|f_{n_k}\|}\right\|
\leq \sum_{k=1}^\infty |\alpha_k|.$$

Let us set
$$r=\frac{\lambda}{\pi \sup_{k\in\N} \|f_{n_k}\|} \qquad \text{and} \qquad \theta = (1-\eps)\, r \inf_{k\in\N} \|f_{n_k}\|.$$
Then
$$\theta = (1-\eps)\frac{\lambda}{\pi}\, \frac{\inf_{k\in\N} \|f_{n_k}\|}{\sup_{k\in\N} \|f_{n_k}\|}
\geq (1-\eps) \frac{\lambda}{\pi}\, \frac{\inf_{n\in\N} \|f_n\|}{\sup_{n\in\N} \|f_n\|}
\geq (1-\eps)^2 \frac{\lambda}{\pi}.$$
Without loss of generality we can assume that $(f_{n_k})=(f_n)$.
Then $\big(\frac{f_n}{\|f_n\|}\big)_n$ is a~basic sequence consisting of selfadjoint elements which satisfies
$$r \sum_{k=1}^\infty |\alpha_k| \leq \left\|\sum_{k=1}^\infty \alpha_k \frac{f_k}{\|f_k\|}\right\| \leq \sum_{k=1}^\infty |\alpha_k|, \quad (a_k)\in\ell^1,$$
that is (36) of \cite{pfitzner} (where $a_k'=f_k$).
By Pfitzner's proof of \cite[Theorem 1]{pfitzner} we obtain a~sequence $(x_k)$ in $A$ and a subsequence $(f_{n_k})$ of $(f_n)$ for which (35) of \cite{pfitzner} is valid (where $a_{n_k}'=f_{n_k}$), i.e. $x_k$ are selfadjoint elements in $B_A$ such that $x_i x_j = 0$, $i,j\in\N$, $i\not=j$, and $\left|f_{n_k}(x_k)\right| > \theta \geq (1-\eps)^2 \frac{\lambda}{\pi}$, $k\in\N$. This completes the proof of the claim.
\end{proof}
Let $(x_k)$ and $(x_{n_k}')$ be sequences obtained by the claim. Since $|x_{n_k}'(x_k)| > (1-\eps)^2 \frac{\lambda}{\pi}$, $k\in\N$, we have
$$\limsup_{k\to\infty} \sup_{x'\in K} |x'(x_k)| \geq (1-\eps)^2 \frac{\lambda}{\pi}.$$
But $\sum x_k$ is a wuC series in $A$ satisfying $\sup_{x'\in B_{A'}} \sum |x'(x_k)| \leq 1$.
Indeed, all $x_k$ are contained in a commutative subalgebra $B$ of $A$, which can be identified with the space $C_0(\Omega)$ for some $\Omega$ by the Gelfand representatiton. Then $x_k$, $k\in\N$, are real continuous functions on $\Omega$ with $\|x_k\| = \sup_{\xi\in\Omega}|x_k(\xi)| \leq 1$ and $\{x_i\not=\nobreak0\}\cap\{x_j\not=\nobreak0\}=\nobreak\emptyset$, $i\not=j$. Let $x'\in A'$, and let us set $\mu = x'\restriction_B \in B' = C_0(\Omega)' = \mathcal{M}(\Omega)$. For each $N\in\N$ we get
$$\begin{aligned}
\sum_{k=1}^N |x'(x_k)| &= \sum_{k=1}^N |\mu(x_k)| = \sum_{k=1}^N \left|\int_{\Omega} x_k\, \mathrm{d}\mu \right|
\leq \sum_{k=1}^N \int_{\{x_k\not=0\}} |x_k|\, \mathrm{d}|\mu| \\
&\leq \int_{\Omega} 1\, \mathrm{d}|\mu| = \|\mu\| \leq \|x'\|.
\end{aligned}$$
Therefore $\sup_{x'\in B_{A'}} \sum_{k=1}^\infty |x'(x_k)| \leq 1$.

We thus obtain $\eta(K) \geq (1-\eps)^2 \frac{\lambda}{\pi}$. Since $\eps>0$ and $\lambda<\wck_{A'}(K)$ were chosen arbitrarily, it follows that $\eta(K) \geq \frac{1}{\pi} \wck_{A'}(K)$, which completes the proof.
\end{proof}

\begin{remark}
\begin{sloppypar}
It is not clear whether $C^*$-algebras have also the property \qos{V}. From \cite[Theorem 4.1]{qpel} it follows that the answer is affirmative for commutative \mbox{$C^*$-algebras}. In fact we do not know any example of a Banach space with the property \q{V} but not \qos{V}. Regarding the property \qo{V}, we know from \cite[Proposition 4.3]{qpel} that some (commutative) $C^*$-algebras enjoy this property and some do not.
\end{sloppypar}
\end{remark}

The following theorem is a kind of ``real version'' of Theorem \ref{qpf}.

\begin{theorem}
\label{sa}
Let $A$ be a $C^*$-algebra. Then the space $A_{\sa}$ has the property \q{V}, more precisely, for every bounded $K\subset A'_{\sa}$
\begin{equation}
\label{qp-sa}
\wck_{A'}(K) \leq \eta(K).
\end{equation}
\end{theorem}
\begin{proof}
The proof is analogous to the previous one, it suffices to use real versions of the theorems that have allowed us to prove Theorem \ref{qpf}.
Let us sketch it briefly.

Consider a bounded set $K\subset A'\sa$ with $\wck_{A'\sa}(K) > 0$ and an arbitrary $\lambda\in(0,\wck_{A'\sa}(K))$. We find $(f_n)$ in $K$ such that
$$\dist\left(\clust_{((A'\sa)'',w^*)}(f_n),A'\sa\right) > \lambda.$$
Since $A'$ is L-embedded, the real version of $A'$ (let us denote it by $(A')_{\er}$) is also L-embedded. But $(A')\sa$ is a 1-complemented subspace of $(A')_{\er}$ and is therefore L-embedded by \cite[Proposition IV.1.5]{mideals}. We thus get
$$\tilde{\delta}(f_n) > 2\lambda$$
from \cite[Theorem 1]{kps}.
Let us fix $\eps>0$. By passing to a subsequence we arrange that
$$\frac{\inf_{n\in\N} \|f_n\|}{\sup_{n\in\N} \|f_n\|} > 1-\eps.$$
By the real version of the quantitative Rosenthal's $\ell^1$--theorem \cite[Theorem 3.2]{behrends} the sequence $(f_n)$ admits $\lambda$--$\ell^1$--blocks, which yields a subsequence $(f_{n_k})$ of the sequence $(f_n)$ that for every $(\alpha_n) \in \ell^1$ satisfies
$$\frac{\lambda}{\sup_{k\in\N} \|f_{n_k}\|} \sum_{k=1}^\infty |\alpha_k|
\leq \left\|\sum_{k=1}^\infty \alpha_k \frac{f_{n_k}}{\|f_{n_k}\|}\right\|
\leq \sum_{k=1}^\infty |\alpha_k|.$$
Then we proceed exactly as in the proof of Theorem \ref{qpf} to obtain the desired conclusion.
\end{proof}

\section{Relation to the Grothendieck property}
\label{groth}
Let us remind that a Banach space $X$ has the Grothendieck property if every weak$^*$ convergent sequence in its dual is weakly convergent. It is well known that for dual Banach spaces the property (V) implies the Grothendieck property. In this section we prove that this implication holds even for suitable quantitative versions of these properties.

One possible quantification of the Grothendieck property has already been studied in \cite{qgroth} and \cite{lechner}. Let us remind the definition: Let $c>0$. A Banach space $X$ is \emph{$c$-Grothendieck}\/ if
\begin{equation}
\label{c-groth}
\delta(x_n') \leq c \cdot \dws(x_n')
\end{equation}
whenever $(x_n')$ is a bounded sequence in $X'$.

A Banach space $X$ has the Grothendieck property if and only if for every sequence $(x_n')$ in $X'$ the following implication holds:
$$(x_n') \text{ is weak}^* \text{ Cauchy} \ \Rightarrow \ (x_n') \text{ is weakly Cauchy}.$$
The inequality \eqref{c-groth} quantifies this implication. But we can look at the Grothendieck property also in another way: $X$ has the Grothendieck property if and only if every sequence $(x_n')$ in $X'$ satisfies the implication
$$(x_n') \text{ is weak}^* \text{ Cauchy} \ \Rightarrow \ \{x_n':\,n\in\N\} \text{ is a relatively weakly compact set.}$$
If we replace this implication by an inequality
$$\wck_{X'}\big(\{x_n':\,n\in\N\}\big) \leq c\cdot \dws(x_n')$$
where $c>0$ is some constant not depending on $(x_n')$, we obtain another quantitative version of the Grothendieck property. We will prove that all dual Banach spaces with the property \q{V} have this kind of quantitative Grothendieck property (see Corollary \ref{qv-qgroth}). We do not know whether the latter quantitative Grothendieck property implies the former one (with a larger constant).


\begin{theorem}
\label{eta-dws}
Let $X$ be a Banach space. Then for every bounded sequence $(x_n'')$ in $X''$
$$\eta\big(\{x_n'':\, n\in\N\}\big) \leq \tfrac12\dws(x_n'').$$
\end{theorem}
\begin{proof}
Let $(x_n'')$ be a bounded sequence in $X''$. The case $\eta\big(\{x_n'':\, n\in\N\}\big)=0$ is trivial. Suppose that $\eta\big(\{x_n'':\, n\in\N\}\big)>0$ and fix $\delta\in\left(0,\eta\big(\{x_n'':\, n\in\N\}\big)\right)$. Let us find $\eps>0$ such that $\eta\big(\{x_n'':\, n\in\N\}\big)>\delta+\eps$. By the definition of the quantity $\eta$ we can find a wuC series $\sum_{k=1}^\infty x_k'$ in $X'$ with $\sup_{x''\in B_{X''}} \sum_{k=1}^\infty |x''(x_k')| \leq 1$ such that $\limsup_{k\to\infty} \sup_{n\in\N} |x_n''(x_k')|>\delta+\eps$. Since $(x_k')$ is a weakly null sequence, there are subsequences of $(y_n'')$ of $(x_n'')$ and $(y_k')$ of $(x_k')$ which for all $n\in\N$ satisfy $|y_n''(y_n')| > \delta+\eps$. The sequence $(y_n')$ is weakly null in $X'$ and $(y_n'')$ is a bounded sequence in $X''$, hence by Simons' extraction lemma \cite[Theorem 1]{simons} there is a~strictly increasing sequence of indices $(n_k)$ such that for all $k\in\N$
$$\sum_{\substack{m\in\N \\ m\not=k}} |y_{n_k}''(y_{n_m}')|<\eps.$$

Let us define
$$\alpha_k =
\begin{cases}
(-1)^k \sgn^{-1}\big(y_{n_k}''(y_{n_k}')\big), &y_{n_k}''(y_{n_k}')\not=0, \\
0, &y_{n_k}''(y_{n_k}') = 0,
\end{cases} \qquad
k\in\N,$$
where $\sgn$ denotes the complex signum function, i.e. $\sgn(z) = \frac{z}{|z|}$, $z\in\C\setminus\{0\}$.
Set $$x' = w^*\!\text{-}\!\lim_{N\to\infty} \sum_{k=1}^N \alpha_k y_{n_k}'\in X'.$$
Then $x'\in B_{X'}$ because for all $x\in B_X$
$$|x'(x)| = \left|\sum_{k=1}^\infty \alpha_k z_{n_k}'(x)\right| \leq \sum_{k=1}^\infty |z_{n_k}'(x)| \leq \sum_{n=1}^\infty |x_n'(x)| \leq \sup_{x''\in B_{X''}} \sum_{n=1}^\infty |x''(x_n')| \leq 1.$$

For each $k\in\N$ even
$$
\begin{aligned}
\operatorname{Re}y_{n_k}''(x') &= \alpha_k y_{n_k}''(y_{n_k}') + \operatorname{Re}\Bigg(\sum_{\substack{m\in\N \\ m\not=k/2}} y_{n_k}''(\alpha_{2m} y_{n_{2m}}')\Bigg) \\
&\qquad- \operatorname{Re}\left(\sum_{m\in\N} y_{n_k}''(\alpha_{2m-1} y_{n_{2m-1}}')\right) \\
&\geq |y_{n_k}''(y_{n_k}')| - \sum_{\substack{m\in\N \\ m\not=k/2}} |y_{n_k}''(y_{n_{2m}}')| - \sum_{m\in\N} |y_{n_k}''(y_{n_{2m-1}}')|  \\
&= |y_{n_k}''(y_{n_k}')| - \sum_{\substack{m\in\N \\ m\not=k}} |y_{n_k}''(y_{n_m}')| \\
&> (\delta+\eps)-\eps = \delta.
\end{aligned}
$$
Analogously, for each $k\in\N$ odd
$$
\begin{aligned}
\operatorname{Re}y_{n_k}''(x') &= \alpha_k y_{n_k}''(y_{n_k}') + \operatorname{Re}\left(\sum_{m\in\N} y_{n_k}''(\alpha_{2m} y_{n_{2m}}')\right) \\
&\qquad- \operatorname{Re}\Bigg(\sum_{\substack{m\in\N \\ m\not=(k+1)/2}} y_{n_k}''(\alpha_{2m-1} y_{n_{2m-1}}')\Bigg) \\
&\leq -|y_{n_k}''(y_{n_k}')| + \sum_{\substack{m\in\N \\ m\not=k}} |y_{n_k}''(y_{n_m}')| \\
&< -(\delta+\eps)+\eps = -\delta.
\end{aligned}
$$
Therefore
$$\inf_{n\in\N} \sup_{k,l\geq n} |y_{n_k}''(x') - y_{n_l}''(x')| \geq \inf_{n\in\N} \sup_{k,l\geq n} \big|\operatorname{Re}\big(y_{n_k}''(x') - y_{n_l}''(x')\big)\big| \geq 2\delta.$$
It follows that $\dws(x_n'') \geq \dws(y_{n_k}'') \geq 2\delta$. Since $\delta<\eta\big(\{x_n'':\, n\in\N\}\big)$ was chosen arbitrarily, we obtain the desired inequality.
\end{proof}

\begin{corollary}
\label{qv-qgroth}
Let $X$ be a Banach space and $C>0$. Suppose that each bounded $K\subset X''$ satisfy
\begin{equation}
\label{c-qv}
\wck_{X''}(K) \leq C\cdot \eta(K)
\end{equation}
(i.e. $X'$ enjoys the property \q{V}).
Then for every bounded sequence $(x_n'')$ in $X''$
$$\wck_{X''}\big(\{x_n'':\, n\in\N\}\big) \leq \tfrac{1}{2}C\cdot \dws(x_n'').$$
\end{corollary}
\begin{proof}
It suffices to combine the previous theorem with the inequality \eqref{c-qv} applied to $K=\{x_n'':\, n\in\N\}$.
\end{proof}

\begin{corollary}
Let $A$ be a von Neumann algebra. Then $A$ has a quantitative version of the Grothendieck property -- more precisely, for every bounded sequence $(x_n')$ in $A'$
$$\wck_{A'}\big(\{x_n':\, n\in\N\}\big) \leq \tfrac{1}{2}\pi\, \dws(x_n').$$
\end{corollary}
\begin{proof}
Since every von Neumann algebra is a $C^*$-algebra and a dual Banach space, the assertion follows from Theorem \ref{qpf} and Corollary \ref{qv-qgroth}.
\end{proof}

\section*{Acknowledgement}
The author wishes to express her gratitude to the referee, who suggested how to improve Theorem \ref{eta-dws} and shorten its proof. She also acknowledges many suggestions and comments of her supervisor Ond\v rej Kalenda during the preparation of the paper.

\bibliography{c-star-alg-q-pel-corr}
\bibliographystyle{plain}

\end{document}